\documentclass[12pt,reqno,a4paper]{amsart}
\usepackage{amsmath,amsfonts,amssymb,amsthm,amsxtra}
\usepackage{epsfig}
\newtheorem{thm}{Theorem}[section]

\newtheorem{corollary}{Corollary}
\newtheorem{theorem}{Theorem}

\newtheorem{cor}[thm]{Corollary}
\newtheorem{proposition}{Proposition}
\newtheorem{example}{Example}

\newcommand{\remark}[1]{\noin{{\bf Remark.}} }

\def\K#1#2{\displaystyle{\mathop K\limits_{#1}^{#2}}}

\setbox0=\hbox{$+$}
\newdimen\plusheight
\plusheight=\ht0
\def\+{\;\lower\plusheight\hbox{$+$}\;}

\setbox0=\hbox{$-$}
\newdimen\minusheight
\minusheight=\ht0
\def\-{\;\lower\minusheight\hbox{$-$}\;}

\setbox0=\hbox{$\cdots$}
\newdimen\cdotsheight
\cdotsheight=\plusheight
\def\cds{\lower\cdotsheight\hbox{$\cdots$}}

\begin{document}
\title[Continued Fractions with Many Limits]{Continued Fractions and Generalizations
with Many Limits: A Survey.}
\author{Douglas Bowman
}
\address{Northern Illinois University,
   Mathematical Sciences,
   DeKalb, IL 60115-2888
} \email{ bowman@math.niu.edu}
\author{James Mc Laughlin
}
\address{Mathematics Department,
 Anderson Hall,
West Chester University, West Chester, PA 19383 }
\email{jmclaughl@wcupa.edu }
\thanks{The first author's research was partially supported by NSF grant
DMS-0300126.}

\date{July 31, 2006}

\begin{abstract}
There are infinite processes (matrix products, continued
fractions, $(r,s)$-matrix continued fractions,  recurrence
sequences) which, under certain circumstances, do not converge but
instead diverge in a very predictable way.

We give a survey of results in this area, focusing on recent
results of the authors.
\end{abstract}

\maketitle

\section{Introduction}

Consider the following recurrence:

$$x_{n+1}=\frac {4}{3} -\frac {1}{x_n}.$$
Taking $1/\infty$ to be $0$ and vice versa, then regardless of the initial (real) value of this sequence,
it is an interesting fact that the sequence is dense in $\mathbb{R}$. The proof is illuminating.

Take $x_0=4/3$ and view $x_n$ as $n$'th approximant of the
continued fraction:
\begin{equation}\label{32}
4/3-\frac{1}{4/3}\-\frac{1}{4/3}\-\frac{1}{4/3}\-\frac{1}{4/3}\-
\cds.
\end{equation}
Then, from the standard theorem on the recurrence for convergents of
a continued fraction, the $n$'th numerator and denominator convergents of this
continued fraction, $A_n$ and $B_n$ respectively, must both satisfy the linear
recurrence relation
\[
Y_n=\frac {4}{3} Y_{n-1}-Y_{n-2},
\]
but with different initial conditions.

Now, the characteristic roots of this equation are
$\alpha=2/3+i\sqrt{5}/3$, and $\beta=2/3-i \sqrt{5}/3$. Thus from
the usual formula for solving linear recurrences, the exact
formula for $x_n$ is
\[
x_n=\frac{A_n}{B_n}=\frac{a\alpha^n+b\beta^n}{c\alpha^n+d\beta^n}=\frac{a\lambda^n+b}{c\lambda^n+d},
\]
where $a$, $b$, $c$, and $d$ are some complex constants and $\lambda=\alpha/\beta$.
Notice that $\lambda$ is a number on
the unit circle and is not a root of unity, so that $\lambda^n$ is dense on the unit circle. The
conclusion follows by noting that the linear fractional transformation
\[
z\mapsto\frac{az+b}{cz+d}
\]
must take the unit circle to  $\mathbb{R}$, since the values of the sequence $x_n$ are real.

After seeing this argument, one is tempted to write down the amusing
identity
\[
\mathbb{R}=4/3-\frac{1}{4/3}\-\frac{1}{4/3}\-\frac{1}{4/3}\-\frac{1}{4/3}\-
\cds.
\]
This identity is true so long as one interprets the value of the continued fraction
to be the set of limits of subsequences of its sequence of approximants.

Another motivating example of our work is the following theorem, one of the
oldest in the analytic theory of continued fractions \cite{LW92}:

\begin{thm} (Stern-Stolz)
\label{Stern-Stolz1}
Let the sequence $\{b_{n}\}$ satisfy $\sum |b_{n}| <\infty$. Then
\[
b_{0}+K_{n=1}^{\infty}\frac{1}{b_{n}}
\]
diverges. In fact, for $p=0,1$,
\begin{align*}
&\lim_{n \to \infty}P_{2n+p}=A_{p} \not = \infty,& &\lim_{n \to \infty}Q_{2n+p}=B_{p} \not = \infty,&
\end{align*}
and
\[
A_{1}B_{0}-A_{0}B_{1} = 1.
\]
\end{thm}

The Stern-Stolz theorem gives a general class of continued fractions each
of which tend
to two different limits, respectively $A_0/B_0$, and $A_1/B_1$. Here and
throughout we assume the limits for continued fractions are in $\widehat{\mathbb{C}}$.
This makes sense because continued fractions can be viewed as the composition
of linear fractional transformations and such functions have $\widehat{\mathbb{C}}$
as their natural domain and codomain.

Before leaving the Stern-Stolz theorem, we wish to remark that although the theorem is
usually termed a ``divergence theorem'', this terminology is a bit misleading; the theorem
actually shows that although the continued fractions of this form diverge, they do so by tending
to two limits in a precisely controlled way. In this paper we study extensions of this
phenomenon and investigate just how far one can go in this direction. Thus, although throughout
this paper we refer to certain of our results as ``divergence'' theorems, most of them
actually give explicit results about convergent subsequences.

A special case of the Stern-Stolz theorem gives a result on
the famous Rogers-Ramanujan continued
fraction:

\begin{equation}\label{RR}
 1+
\frac{q}{1}
\+
\frac{q^{2}}{1}
\+
\frac{q^{3}}{1}
 \+
\frac{q^{4}}{1}
\cds,
\end{equation}

The Stern-Stolz theorem gives that for $|q|>1$ the even and odd
approximants of this continued fraction tend
to two limiting functions. To see this, observe that
by the standard equivalence transformation for continued fractions, (\ref{RR})
is equal to
 \[
 1+
\frac{1}{1/q}
\+
\frac{1}{1/q}
\+
\frac{1}{1/q^{2}}
 \+
\frac{1}{1/q^{2}}
\cds
\+
\frac{1}{1/q^{n}}
 \+
\frac{1}{1/q^{n}}
\cds
.
\]

The Stern-Stolz theorem, however does not apply to the
following continued fraction given by Ramanujan:
\begin{equation}\label{R3}
\frac{-1}{1+q}
\+
\frac{-1}{1+q^2}
\+
\frac{-1}{1+q^3}
\+
\cds .
\end{equation}

Recently in  \cite{ABSYZ02} Andrews, Berndt, {\it{et al.}}  proved a claim
made
by Ramanujan in his lost notebook (\cite{S88}, p.45) about (\ref{R3}). To
describe Ramanujan's claim, we first need some notation.
 Throughout take $q\in\mathbb{C}$
with $|q|<1$. The following standard notation for $q$-products
will also be employed:
\begin{align*}
&(a)_{0}:=(a;q)_{0}:=1,& &
(a)_{n}:=(a;q)_{n}:=\prod_{k=0}^{n-1}(1-a\,q^{k}),& & \text{ if }
n \geq 1,&
\end{align*}
and
\begin{align*}
&(a;q)_{\infty}:=\prod_{k=0}^{\infty}(1-a\,q^{k}),& & |q|<1.&
\end{align*}
Set $\omega = e^{2 \pi i/3}$. Ramanujan's claim was that, for
$|q|<1$,
{\allowdisplaybreaks
\small{
\begin{equation}\label{3lim1}
\lim_{n \to \infty}
 \left (
\frac{1}{1}
 \-
\frac{1}{1+q}
\-
\frac{1}{1+q^2}
\-
\cds
\-
\frac{1}{1+q^n +a}
\right )
=
-\omega^{2}
\left (
\frac{\Omega - \omega^{n+1}}{\Omega - \omega^{n-1}}
\right ).
\frac{(q^{2};q^{3})_{\infty}}{(q;q^{3})_{\infty}},
\end{equation}
}
}
where
{\allowdisplaybreaks
\begin{equation*}
\Omega :=\frac{1-a\omega^{2}}{1-a \omega}
\frac{(\omega^{2}q,q)_{\infty}}{(\omega q,q)_{\infty}}.
\end{equation*}
}
Ramanujan's notation is confusing, but what his claim means is
that the limit exists as $n \to \infty$ in each of the three
congruence classes modulo 3, and that the limit is given by the
expression on the right side of (\ref{3lim1}). Also, the appearance
of the variable $a$ in this formula is a bit of a red herring; from elementary
properties of continued fractions, one can derive the result for
general $a$ from the $a=0$ case.

Now \eqref{32} is different from the other examples in that it has subsequences
of approximants tending to infinitely many limits. Nevertheless,
all of the examples above, including \eqref{32}, are
 special cases of a general result on continued fractions (Theorem \ref{T1} below). To deal with
both of these situations we introduce the notion of the {\it limit set}
of a sequence.

The {\it limit set} of the sequence is defined
to be the set of all limits of convergent
subsequences. Limit sets should not be confused with
sets of limit points. Thus, for example, the sequence
$\{1,1,1,\dots\}$ has limit set $\{1\}$ although the set of limit (accumulation)
points of the set of values of the sequence is empty.  Limit sets need to be
introduced so that sequences with constant subsequences will have the values
of these subsequences included among the possible limits. Certain periodic
continued fractions have this property. To avoid
confusion we designate the limit set of a sequence $\{s_n\}_{n\ge
1}$ by $l.s.(s_n)$.

Our initial research \cite{BMcL05} dealt with cases in which the
limit set was finite. In  \cite{BMcL06} we extended our methods to
give a uniform treatment of finite and infinite cases. In fact, in
\cite{BMcL06}, we studied asymptotics for approximants for
infinite matrix products, continued fractions, and recurrence
relations of Poincar{\'e} type. Limit set information easily
follows from the asymptotics.

In the papers \cite{BMcL05} and \cite{BMcL06}, the authors studied
limit sets in the specific context  of sequences of the form
\[
f\left(\prod_{i=1}^nD_i\right),
\]
where $D_i$ is a sequence of complex matrices and $f$ is a function with values in
some compact metric space.

\section{Definitions, Notation and Terminology} \label{DNT}

Limit set equalities in this paper arise  from the
situation
\[\lim_{n\to\infty}d(s_n,t_n)=0
\]
in some metric space $(X,d)$.
Accordingly,
it makes sense to define the equivalence relation $\sim$ on sequences in $X$ by
$\{s_n\}\sim\{t_n\}\iff \lim_{n\to\infty}d(s_n,t_n)=0$. In this situation we
refer to sequences $\{s_n\}$ and $\{t_n\}$ as being asymptotic to each other.
Abusing notation, we often write $s_n\sim t_n$ in
place of $\{s_n\}\sim\{t_n\}$. More generally, we frequently write sequences without braces
when it is clear from context that we are speaking of a sequence, and not the $n$th term.

Let $M_d(\mathbb{C})$ denote the set of $d\times d$ matrices of
complex numbers topologised  using the $l_{\infty}$ norm, denoted
by $||\cdot||$. Let $I$ denote the identity matrix. When we use
product notation for matrices, the product is taken from left to
right; thus
\[
\prod_{i=1}^n A_i:=A_1A_2\cdots A_n.
\]

An infinite continued fraction
\begin{equation}\label{cfgen}
K_{n=1}^{\infty}\frac{a_{n}}{b_{n}}:=\frac{a_{1}}{b_{1}} \+
 \frac{a_{2}}{b_{2}}
\+
 \frac{a_{3}}{b_{3}}
\+\,\cds
\end{equation}
is said to converge if
\begin{equation*}
\lim_{n \to \infty}\frac{a_{1}}{b_{1}} \+
 \frac{a_{2}}{b_{2}}
\+
 \frac{a_{3}}{b_{3}}
\+\,\cds \+
 \frac{a_{n}}{b_{n}}
\end{equation*}
exists in $\widehat{\mathbb{C}}$. Let $\{\omega_{n}\}$ be a sequence
of complex numbers. If
\begin{equation*}
\lim_{n \to \infty}\frac{a_{1}}{b_{1}} \+
 \frac{a_{2}}{b_{2}}
\+
 \frac{a_{3}}{b_{3}}
\+\,\cds \+
 \frac{a_{n}}{b_{n}+\omega_{n}}
\end{equation*}
exist, then this limit is called the {\it {modified limit of
$K_{n=1}^{\infty}a_{n}/b_{n}$ with respect to the sequence
$\{\omega_{n}\}$}}. Detailed discussions of modified continued
fractions as well as further pointers to the literature are given
in \cite{LW92}.

We follow the common convention in analysis of denoting the group
of points on the unit circle by $\mathbb{T}$, or by
$\mathbb{T_\infty}$, and its subgroup of roots of unity of order
$m$, $m$ finite, by $\mathbb{T}_{m}$.  (Note: $\mathbb{T_\infty}$
often denotes the group of all roots of unity; here it denotes the
whole circle group.)


\section{Theorems of Infinite Matrix Products}
The classic theorem on the convergence of infinite products of
matrices seems first to have  been given clearly by Wedderburn
\cite{Wedderburn1}.
\begin{proposition}{(Wedderburn \cite{Wedderburn1,Wedderburn})}\label{Wedderburn} Let $A_i\in M_d(\mathbb{C})$ for $i\ge1$. Then
$\sum_{i\ge1}||A_i||<\infty$ implies that $\prod_{i\ge1}(I+A_i)$
converges in $M_d(\mathbb{C})$.
\end{proposition}
In \cite{BMcL05}, our initial motivation was to generalize  the
Ramanujan continued fraction with three limits \eqref{3lim1} to a
continued fraction with $m$ limits, $m \geq 3$. This led us to
consider infinite sequences of matrices converging to $2 \times 2$
matrices with eigenvalues which were distinct roots of unity, and
to examine the divergence of the corresponding infinite matrix
product.

This in turn led us to consider the more general case of infinite
sequences of $p\times p$ matrices, $p \geq 2$, with similar
properties. In \cite{BMcL05} we proved the following result.
\begin{proposition}\label{tm}
Let $p\geq 2$ be an integer and let $M$ be a $p \times p$ matrix
that is diagonalizable and whose eigenvalues are roots of unity.
Let $I$ denote the $p \times p$ identity matrix and let $m$ be the
least positive integer such that
\[
M^{m}=I.
\]
For a $p \times p$ matrix $G$, let
\[
||G||_{\infty}= \max_{1\leq i,j\leq p}|G^{(i,j)}|,
\]
where $G^{(i,j)}$ denotes the element of $G$ in row $i$ and column
$j$. Suppose $\{D_{n}\}_{n=1}^{\infty}$ is a sequence of matrices
such that
\[
\sum_{n=1}^{\infty}||D_{n}-M||_{\infty}< \infty.
\]
 Then
\[F:=\lim_{k \to \infty} \prod_{n=1}^{km}D_{n}
\]
 exists. Here the matrix product means either $D_{1}D_{2} \dots$ or $\dots D_{2}D_{1}$.
Further, for each $j$, $0 \leq j \leq m-1$,
\[
\lim_{k \to \infty} \prod_{n=1}^{km+j}D_{n} = M^{j}F \text{ or }
FM^{j},
\]
depending on whether the products are taken to the left or right.
\end{proposition}

A natural progression was to replace the matrix $M$ in the
proposition above with a sequence of matrices $\{M_i\}$. In
\cite{BMcL06} we proved the following result.
\begin{theorem}\label{DM}
Suppose $\{M_i\}$ and $\{D_i\}$ are sequences of complex matrices
such that the two sequences (for $\epsilon = \pm1$)
\begin{equation}\label{cond1}
\left\|\left(\prod_{i=1}^nM_i\right)^{\epsilon}\right\|
\end{equation}
are bounded and
\begin{equation}\label{cond2}
\sum_{i\ge 1}\|D_i-M_i\|<\infty.
\end{equation}
Then
\begin{equation}
F:=\lim_{n\to\infty}\left(\prod_{i=1}^n
D_i\right)\left(\prod_{i=1}^n M_i\right)^{-1}
\end{equation}
exists and $\det(F)\neq 0$ if and only if $\det(D_i) \neq 0$ for
all $i\geq 1$.

As sequences
\begin{equation}\label{asym}
\prod_{i=1}^n D_i\sim F\prod_{i=1}^n M_i.
\end{equation}
More generally, let $f$ be a  continuous function from the domain
\[
\overline{\left\{F\prod_{i=1}^n M_i:n\ge h\right\}}\cup
\bigcup_{n\ge h} \left\{\prod_{i=1}^n D_i\right\},
\]
for some integer $h\ge1$, into a metric space $G$. Then the domain
of $f$ is compact in $M_d(\mathbb{C})$ and $f(\prod_{i=1}^n
D_i)\sim f(F\prod_{i=1}^n M_i)$. Finally
\begin{equation}
l.s.\left(\prod_{i=1}^n D_i\right)=l.s.\left(F\prod_{i=1}^n
M_i\right),
\end{equation}
and
\begin{equation}
l.s.\left(f\left(\prod_{i=1}^n
D_i\right)\right)=l.s.\left(f\left(F\prod_{i=1}^n
M_i\right)\right).
\end{equation}
\end{theorem}
Theorem \ref{DM} had several interesting applications to certain
classes of continued fractions, recurrence sequences, and $(r,s)$-matrix
continued fractions.

\section{Theorems on Continued Fractions} \label{A}

We begin by stating our general theorem on the asymptotics and
 limit sets of the sequence of
approximants of a class of continued fractions.
The theorem shows that the limit set is a circle (or a
finite subset of a circle) on the Riemann sphere. When the limit
set is a circle, although the set of approximants approaches all
of its points, the approximants usually do not do so in a uniform way.

The following theorem concerns the continued fraction
\begin{equation}\label{basiccf}
 \frac{- \alpha\beta+q_{1}}{\alpha+\beta+p_{1}}
\+ \frac{-\alpha\beta+q_{2}}{\alpha+\beta+p_{2}} \+
 \cds
\+\frac{-\alpha\beta+q_{n}}{\alpha+\beta+p_{n}},
\end{equation}
where the sequences $p_n$ and $q_n$ are absolutely summable and
the constants $\alpha\neq\beta$ are points on the unit circle.

\begin{thm}\label{T1}
Let $\{p_{n}\}_{n\geq 1}$, $\{q_{n}\}_{n \ge 1}$ be complex
sequences satisfying
\begin{align*}
 &\sum_{n=1}^{\infty}|p_{n}|<\infty,& &\sum_{n=1}^{\infty}|q_{n}|<\infty.&
\end{align*}
Let $\alpha$  and $\beta$ satisfy $|\alpha|=|\beta|=1$, $\alpha
\not = \beta$ with the order of $\lambda=\alpha/\beta$ in
$\mathbb{T}$ being $m$ $($where $m$ may be infinite$)$. Assume
that $q_{n}\not = \alpha\beta$ for any $n \geq 1$. Put
\[
f_n(w):= \frac{- \alpha\beta+q_{1}}{\alpha+\beta+p_{1}}
\+ \frac{-\alpha\beta+q_{2}}{\alpha+\beta+p_{2}} \+
 \cds
\+\frac{-\alpha\beta+q_{n}}{\alpha+\beta+p_{n}+w},
\]
so that $f_n:=f_n(0)$ is the sequence of approximants of the continued fraction \eqref{basiccf}.
Then $f_n\sim h(\lambda^{n+1})$ so that $l.s.\left( f_n\right)= h(\mathbb{T}_m)$, where
{\allowdisplaybreaks
\[
h(z)=\frac {az+b}{cz+d},
\]
with the constants $a,b,c,d\in\mathbb{C}$ given by the (existent) limits
\begin{align}\label{abcddef}
a&=\lim_{n\to\infty}\alpha^{-n}(P_n - \beta
P_{n-1}),\\
 b&=-\lim_{n\to\infty}\beta^{-n}(P_n-\alpha P_{n-1}),\notag\\
c&=\lim_{n\to\infty}\alpha^{-n}(Q_n - \beta
Q_{n-1}),\notag\\
 d&=-\lim_{n\to\infty}\beta^{-n}(Q_n-\alpha Q_{n-1}),\notag
\end{align}
where $P_n$ and $Q_n$ are the $n$th convergents of the continued
fraction \eqref{basiccf}.
Moreover,
\begin{equation}\label{detid}
\det(h)=ad-bc=(\beta-\alpha)\prod_{n=1}^\infty\left(1-\frac{q_n}{\alpha\beta}\right)\neq 0,
\end{equation}
and the following identities  involving modified versions of \eqref{basiccf} hold in ${\widehat{\mathbb{C}}}$:
\begin{align}\label{hinfin}
h(\infty)&=\frac ac\\&=
\lim_{n\to\infty}
\frac{- \alpha\beta+q_{1}}{\alpha+\beta+p_{1}}
\+ \frac{-\alpha\beta+q_{2}}{\alpha+\beta+p_{2}} \+
 \cds
\+\frac{-\alpha\beta+q_{n-1}}{\alpha+\beta+p_{n-1}}
\+\frac{-\alpha\beta+q_{n}}{\alpha+p_{n}};\notag\\
\label{h0}
h(0)&=\frac bd\\& =
\lim_{n\to\infty}
 \frac{- \alpha\beta+q_{1}}{\alpha+\beta+p_{1}}
\+ \frac{-\alpha\beta+q_{2}}{\alpha+\beta+p_{2}} \+
 \cds
\+\frac{-\alpha\beta+q_{n-1}}{\alpha+\beta+p_{n-1}}
\+\frac{-\alpha\beta+q_{n}}{\beta+p_{n}};\notag
\end{align}
and for $k\in\mathbb{Z}$, we have
\begin{align}\notag
h(\lambda^{k+1})&=\frac {a\lambda^{k+1}+b}{c\lambda^{k+1}+d}\\\label{hvalues}&=
\lim_{n\to\infty}
\frac{- \alpha\beta+q_{1}}{\alpha+\beta+p_{1}}
\+ \frac{-\alpha\beta+q_{2}}{\alpha+\beta+p_{2}} \+
\cds
\+\frac{-\alpha\beta+q_{n}}{\alpha+\beta+p_{n}+\omega_{n-k}},
\end{align}
where
\[
\omega_{n}=-\frac{\alpha^{n} - \beta^{n}}{\alpha^{n-1} -
\beta^{n-1}}\in{\widehat{\mathbb{C}}}, \hspace{25pt} n\in\mathbb{Z}.
\]
}
\end{thm}

As a first application, we can get quite precise information about
the divergence behavior of limit-1 periodic continued fractions of
elliptic type (see \cite{LW92} for more on limit-1 periodic
continued fractions of elliptic type).

We consider  the case where the continued fraction
\begin{equation}\label{lcfgen}
\frac{a_{1}}{b_{1}} \+
 \frac{a_{2}}{b_{2}}
\+
 \frac{a_{3}}{b_{3}}
\+\,\cds
\end{equation}
 is a limit 1-periodic continued fraction of
elliptic type and, in addition,
 \[\sum_{n\ge1}|a_n-a|<\infty, \hspace{25pt}
 \sum_{n\ge1}|b_n-b|<\infty,
\]
for some $a, b \in \mathbb{C}$.

 Set
\[
d:=\left | \frac{b+\sqrt{b^2+4a}}{2} \right | =\left |
\frac{b-\sqrt{b^2+4a}}{2} \right |,
\]
and define
\begin{align*}
&\alpha = \frac{b+\sqrt{b^2+4a}}{2d},& &\beta
=\frac{b-\sqrt{b^2+4a}}{2d}.&
\end{align*}
Then $\alpha \not = \beta$, $|\alpha|=|\beta|=1$. Define, for $n
\geq 1$,  $p_n$ and $q_n$ by
\begin{align*}
&a_n=a+p_n,& &b_n=b+q_n.&
\end{align*}
Thus
\[
K_{n=1}^{\infty}\frac{a+q_{n}}{b+p_{n}} = d\,
K_{n=1}^{\infty}\frac{-\alpha \beta+q_{n}/d^2}{\alpha + \beta +
p_{n}/d}.
\]

The second continued fraction satisfies the conditions of Theorem
\ref{T1}. Thus this theorem can be applied to all limit 1-periodic
continued fractions of elliptic type with $\lim_{n\to\infty}a_n=a$
and $\lim_{n\to\infty}b_n=b$, providing
$\sum_{n\ge1}|a_n-a|<\infty $ and $\sum_{n\ge1}|b_n-b|<\infty $.
Of course, it is known that without any restrictions on how the
limit periodic sequences tend to their limits, the behavior can be
quit complicated, see \cite{LW92}.

Next,  we can obtain (up to a factor of $\pm 1$) the numbers $a$,
$b$, $c$, and $d$ in terms of the modified continued fractions and
the product for $\det(h)$ given in Theorem \ref{T1}.

\begin{cor}\label{hzc}
The linear fractional transformation $h(z)$ defined in Theorem
\ref{T1} has the following expression
\[
h(z)=\frac{A(C-B)z+B(A-C)}{(C-B)z+A-C},
\]
where $A=h(\infty)$, $B=h(0)$, and $C=h(1)$. Moreover, the
constants $a$, $b$, $c$, and $d$ in the theorem have the following
formulas
\[
a=sA(C-B),\quad b=sB(A-C),\quad c=s(C-B),\quad d=s(A-C),
\]
where
\[
s=\pm\sqrt{\frac{(\beta-\alpha)\prod_{n=1}^\infty\left(1-\frac{q_n}{\alpha\beta}\right)}{(A-B)(C-A)(B-C)}}.
\]
\end{cor}
It is interesting that the linear fractional transformation which
describes the limit set of the divergent continued fraction
\[
K_{n=1}^{\infty} \frac{-\alpha \beta + q_n}{\alpha+\beta+p_n}
\]
can be described completely in terms of three convergent modified
continued fractions.

Let $\mathbb{T}^\prime$ denote the image of $\mathbb{T}$ under $h$, that
is, the limit set of the sequence $\{f_n\}$.
The main conclusion of the theorem can be expressed by the statement
\begin{equation}\label{fn}
f_n\sim h(\lambda^{n+1}),
\end{equation}
where $h$ is the linear fractional transformation in the theorem.
 It is well
known that when $\lambda$ is not a root of unity, $\lambda^{n+1}$ is uniformly
distributed on
$\mathbb{T}$.  However, the linear fractional transformation $h$
stretches and compresses  arcs of the circle $\mathbb{T}$, so that the distribution of
 $h(\lambda^{n+1})$ in arcs of $\mathbb{T}^\prime$ is no longer uniform.
Thus, although the limit set in the case where $\lambda$ is not a
root of unity is a circle, the concentration of approximants is
not uniform around the circle.

Fortunately, the distribution of approximants is completely
controlled by the known parameters $a$, $b$, $c$, and $d$. The
following corollary gives the points on the limit sets whose
neighborhood arcs have the greatest and least concentrations of
approximants. (We do not take the space here to give a precise definition
of what this means;  interested readers should consult the author's
paper \cite{BMcL06}.)

\begin{cor}\label{cnc}
When $m=\infty$ and $cd\neq 0$, the points on
\[\frac{a\mathbb{T}_m+b}{c\mathbb{T}_m+d}
\]
with the highest and lowest concentrations of approximants are
\[
\frac{\displaystyle{\frac{a}{c}|c|+\frac{b}{d}|d|}}{|c|+|d|}
\qquad\text{and}\qquad\frac{-\displaystyle{\frac{a}{c}|c|+\frac{b}{d}|d|}}{-|c|+|d|},
\]
respectively.
If either $c=0$ or $d=0$, then all points on the limit set have the same concentration. The radius of
the limit set circle in $\mathbb{C}$ is
\[
\left|\frac{\alpha-\beta}{|c|^2-|d|^2}
\prod_{n=1}^{\infty}\left(1-\frac{q_n}{\alpha\beta}\right)\right|.
\]
The limit set is a line in $\mathbb{C}$ if and only if $|c|=|d|$, and in
this case the point of least concentration is $\infty$.
\end{cor}

\begin{cor}\label{linec}
If the limit set of the continued fraction
in \eqref{basiccf} is a line in $\mathbb{C}$, then
the point of highest concentration of approximants in the limit set
is exactly
\[
\frac{h(\infty)+h(0)}{2},
\]
the average of the first two modifications of \eqref{basiccf}
given in Theorem \ref{T1}.
\end{cor}

It is also possible to derive a convergent continued fractions which
have the same limit as the modified continued fractions in Theorem
\ref{T1}. These are given in the following corollary.

\begin{cor}\label{cBM}
Let $\alpha$, $\beta$, $\{p_n\}$, $\{q_n\}$, $h(z)$,
 be as in Theorem \ref{T1}. Then
\begin{multline}\label{BMcf1}
h(\infty)=-\beta +\frac{q_1+\beta p_1}{\alpha +p_1} \+ \frac{(q_1-
\alpha \beta)(q_2+ \beta p_2)}{(\alpha+p_2)(q_1+\beta
p_1)+\beta(q_2+\beta p_2)}\\ \+ K_{n=3}^{\infty} \frac{(q_{n-1}-
\alpha \beta)(q_n+ \beta p_n)(q_{n-1}+ \beta
p_{n-1})}{(\alpha+p_n)(q_{n-1}+\beta p_{n-1})+\beta(q_n+\beta
p_n)},
\end{multline}
\begin{multline}\label{BMcf2}
h(0)=-\alpha +\frac{q_1+\alpha p_1}{\beta +p_1} \+ \frac{(q_1-
\alpha \beta)(q_2+ \alpha p_2)}{(\beta+p_2)(q_1+\alpha
p_1)+\alpha(q_2+\alpha p_2)}\\ \+ K_{n=3}^{\infty} \frac{(q_{n-1}-
\alpha \beta)(q_n+ \alpha p_n)(q_{n-1}+ \alpha
p_{n-1})}{(\beta+p_n)(q_{n-1}+\alpha p_{n-1})+\alpha(q_n+\alpha
p_n)}.
\end{multline}

Let $k \in \mathbb{Z}$ and assume  that $\alpha / \beta$ is not a
root of unity. Set
\[
\omega_{n} = -
\frac{\alpha^{n-k}-\beta^{n-k}}{\alpha^{n-k-1}-\beta^{n-k-1}},
\hspace{25pt}\text{ for } n \geq k':=\max \{3, k+3\}.
\]
 Then
 \begin{multline}\label{bmk}
h(\lambda^{k+1})=\frac{- \alpha\beta+q_{1}}{\alpha+\beta+p_{1}} \+
\cds \+
\frac{-\alpha\beta+q_{k'-1}}{\alpha+\beta+p_{k'-1}} \+
\frac{-\alpha\beta+q_{k'}}{\alpha+\beta+p_{k'}+\omega_{k'} }
\\\+ \frac{ -\alpha\beta+q_{k'+1}-\omega_{k'}(
\alpha+\beta+p_{k'+1}+\omega_{k'+1})}{\alpha+\beta+p_{k'+1}+
\omega_{k'+1} } \+ K_{n=k'+2}^{\infty}\frac{ c_{n} }
   {d_{n} },
\end{multline}
where
\begin{align*}
c_{n}& = \left(q_{n-1}-\alpha  \beta \right)\frac{ -\alpha  \beta
+q_n -\omega_{n-1}
 \left(\alpha
   +\beta+p_n +\omega_{n}\right)
} {-\alpha
   \beta +q_{n-1}-\omega_{n-2}
   \left(\alpha +\beta +p_{n-1}+\omega_{n-1}
   \right)
   }\\
d_{n}&= \alpha+\beta+p_n+\omega_{n}- \omega_{n-2} \frac{-\alpha
\beta+q_n- \omega_{n-1}
 \left(\alpha+\beta+p_n +\omega_{n}
 \right)
} {-\alpha
   \beta +q_{n-1}-\omega_{n-2}
   \left(\alpha +\beta
+p_{n-1}+\omega_{n-1} \right)
   }.
\end{align*}
\end{cor}

Before continuing, we give an example which illustrates some of
the results mentioned above. Let $|p|$, $|q|<1$, and define
\[
G(p,q,\alpha, \beta):= \frac{- \alpha\beta+q}{\alpha+\beta+p} \+
\frac{-\alpha\beta+q^{2}}{\alpha+\beta+p^{2}} \+
 \cds
\+\frac{-\alpha\beta+q^{n}}{\alpha+\beta+p^{n}} \+ \cds .
\]

For $p=0.3$, $q=0.2$, $\alpha=\exp(\imath
\sqrt{11})$,$\beta=\exp(\imath \sqrt{13})$, we use Corollary
\ref{cBM} with $p_n=p^n$, $q_n=q^n$ and $k=-1$ and compute the
limits of the three continued fractions there to find
\begin{align*}
h(\infty)&= 1.13121 + 0.772998 i, \\
h(0)&=  1.20138 + 0.0347473 i,\\
h(1)&=    -0.412160 - 0.486753 i.
\end{align*}

We then apply Corollary \ref{hzc} and compute
\begin{align*}
s&=2.97370 + 0.773678 i, \\
a&=0.581867 + 0.408182 i, \\
b&=-0.670885 - 0.294104 i,\\
c&= 0.518727 +
 0.00637067 i,\\
d&= -0.565036 - 0.228462 i.
\end{align*}

With
\[
h(z) := \frac{a z+b}{c z+d},
\]
we now compare the predicted limit set $h(\mathbb{T})$ with the
sequence of approximants. Figure \ref{fig3} shows the first 3000
approximants of $G(0.3, 0.2, \exp(\imath \sqrt{11}),\exp(\imath
\sqrt{13}))$ and the  circle $h(\mathbb{T})$. We see that the
limit set is exactly what is predicted by Theorem \ref{T1}. The
large dots show the points of highest (top) and lowest (bottom)
points of concentration of approximants, as predicted by Corollary
\ref{cnc}, namely  $1.16911 + 0.374194 i$ and $1.60256 -
  4.18725 i$. We see  that prediction and mathematical fact agree in
  this case also.

\begin{figure}[htbp]
\centering \epsfig{file=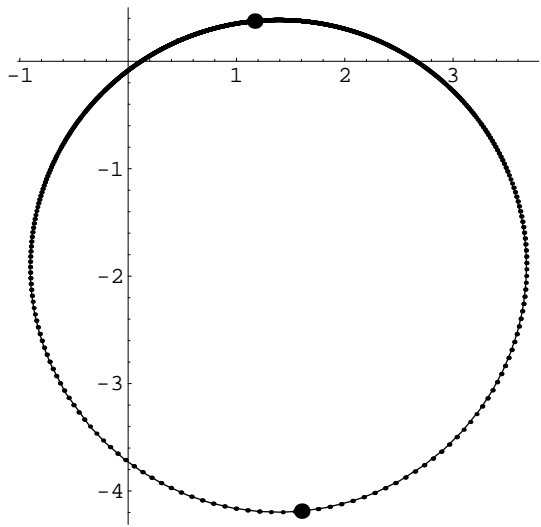, width= 200pt} \caption{The
convergence of $G(0.3, 0.2,\exp(\imath \sqrt{11}),\exp(\imath
\sqrt{13}))$. } \label{fig3}
\end{figure}

We next consider an example where $\alpha/\beta$ is a root of
unity, so that the limit set is finite. We proceed as above to
compute $h(z)$ (details omitted).  Figure \ref{fig4} shows the
first 3000 approximants of $G(0.3, 0.2,\exp(\imath
\sqrt{11}),\exp(\imath (\sqrt{11}+2 \pi/17)))$ and its convergence
to the 17 limit points, together with the  circle $h(\mathbb{T})$.

\begin{figure}[htbp]
\centering \epsfig{file=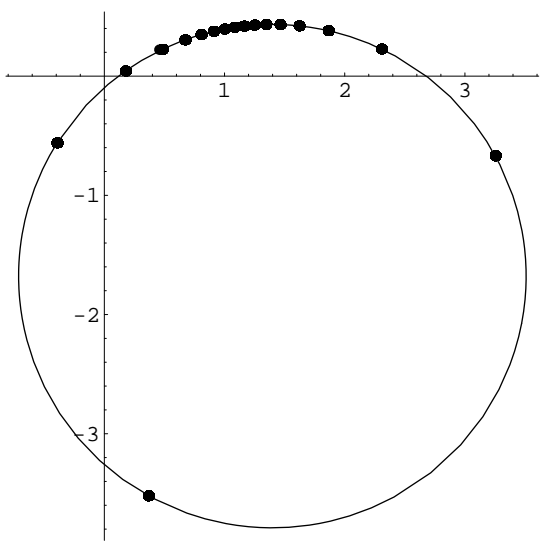, width= 200pt} \caption{The
convergence of $G(0.3, 0.2,\exp(\imath \sqrt{11}),\exp(\imath
(\sqrt{11}+2 \pi/17)))$.} \label{fig4}
\end{figure}

Figure \ref{fig5} shows the image of all seventeen 17th roots of
unity under $h$. Once again the actual limit set and the predicted
limit set agree perfectly.

\begin{figure}[htbp]
\centering \epsfig{file=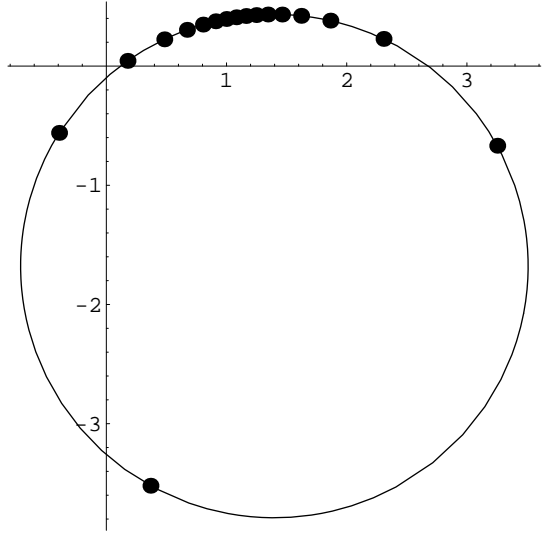, width= 200pt} \caption{The image
of the seventeen 17th roots of unity under $h$.} \label{fig5}
\end{figure}

Lastly, we consider the continued fraction from the beginning of
the paper, $K_{n=1}^{\infty}-1/(4/3)$. If we follow the same kind
of analysis as above, we find that
\begin{align*}
a&=-2/3 + \sqrt{5}/3 i, \\
b&=2/3 + \sqrt{5}/3 i,\\
c&= 1,\\
d&= -1.
\end{align*}
Corollary \ref{linec} predicts that the highest concentration of
approximants occurs at $(a/c+b/d)/2=-2/3$. Figure \ref{fig6} shows
the distribution of the first 1200 approximants of the continued
fraction (about 100 extreme values were omitted), once again
showing agreement with the theory.
\begin{figure}[htbp]
\centering \epsfig{file=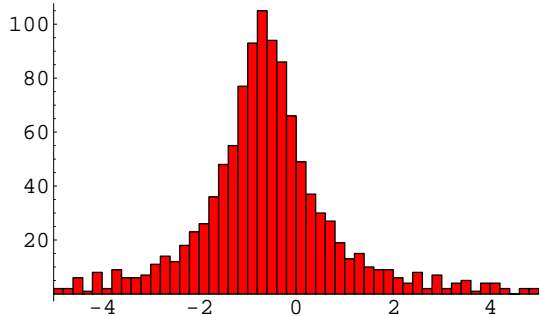, width= 200pt} \caption{The
distribution of the first 1200 approximants of
$K_{n=1}^{\infty}-1/(4/3)$.} \label{fig6}
\end{figure}

\subsection{An Infinite Family of Divergence Theorems}
An interesting special case of Theorem \ref{T1} occurs when
$\alpha$ and $\beta$ are distinct $m$-th roots of unity ($m \geq
2$). In this situation the continued fraction
\[
\frac{- \alpha\beta+q_{1}}{\alpha+\beta+p_{1}} \+
\frac{-\alpha\beta+q_{2}}{\alpha+\beta+p_{2}} \+
\frac{-\alpha\beta+q_{3}}{\alpha+\beta+p_{3}} \+
\frac{-\alpha\beta+q_4}{\alpha+\beta+p_{4}} \+ \cds\]
becomes limit
periodic and the sequences of approximants  in the $m$ different
arithmetic progressions modulo $m$ converge. The corollary below,
which is also proved in \cite{BMcL05}, is an easy consequence of
Theorem \ref{T1}.

\begin{cor}\label{c1}
Let $\{p_{n}\}_{n\geq 1}$, $\{q_{n}\}_{n \ge 1}$ be complex
sequences satisfying
\begin{align*}
 &\sum_{n=1}^{\infty}|p_{n}|<\infty,& &\sum_{n=1}^{\infty}|q_{n}|<\infty.&
\end{align*}
Let $\alpha$  and $\beta$ be distinct roots of unity and let $m$ be
the least positive integer such that $\alpha^m=\beta^{m}=1$ . Define
{\allowdisplaybreaks
\begin{equation*}
G:= \frac{- \alpha\beta+q_{1}}{\alpha+\beta+p_{1}} \+
\frac{-\alpha\beta+q_{2}}{\alpha+\beta+p_{2}} \+
\frac{-\alpha\beta+q_{3}}{\alpha+\beta+p_{3}} \+ \cds.
\end{equation*}
} Let $\{P_{n}/Q_{n}\}_{n=1}^{\infty}$ denote the sequence of
approximants of $G$. If $q_{n}\not = \alpha\beta$ for any $n \geq
1$, then $G$ does not converge. However, the sequences of numerators
and denominators in each of the $m$ arithmetic progressions modulo
$m$ do converge. More precisely,
 there exist complex numbers $A_{0}, \dots  , A_{m-1}$ and $B_{0}, \dots  , B_{m-1}$
such that, for $0 \leq i< m$, {\allowdisplaybreaks
\begin{align}\label{ABlim}
&\lim_{k \to \infty} P_{m\,k+i}=A_{i}, & &\lim_{k \to \infty}
Q_{m\,k+i}=B_{i}.&
\end{align}
Extend the sequences $\{A_i\}$ and $\{B_i\}$ over all integers by
making them periodic modulo $m$ so that (\ref{ABlim}) continues to
hold. Then for integers $i$,
\begin{equation}\label{lol}
A_i=\left(\frac{A_1-\beta A_0}{\alpha-\beta}\right)\alpha^i
+\left(\frac{\alpha A_0-A_1}{\alpha-\beta}\right)\beta^i,
\end{equation}
and
\begin{equation}\label{lolb}
B_i=\left(\frac{B_1-\beta B_0}{\alpha-\beta}\right)\alpha^i
+\left(\frac{\alpha B_0-B_1}{\alpha-\beta}\right)\beta^i.
\end{equation}

Moreover,
\begin{equation}\label{detAB}
A_iB_{j}-A_{j}B_i=-(\alpha\beta)^{j+1}
\frac{\alpha^{i-j}-\beta^{i-j}}{\alpha-\beta}
\prod_{n=1}^{\infty}\left(1-\frac{q_{n}}{\alpha\beta}\right).
\end{equation}
} Put  $\alpha := \exp( 2 \pi i a/m)$,  $\beta := \exp (2 \pi i
b/m)$, $0 \leq a <b <m$, and $r:=m/\gcd(b-a,m)$. Then $G$ has $r$
distinct limits in $\widehat{\mathbb{C}}$ which are given by $A_j/B_j$,
$1\le j\le r$. Finally, for $k \geq 0$ and $1 \leq j \leq r$,
{\allowdisplaybreaks
\begin{equation*}
\frac{A_{j+kr}}{B_{j+kr}}=\frac{A_{j}}{B_{j}}.
\end{equation*}
}
\end{cor}

The number $r$ occuring in this theorem is just the number of distinct limits to
which the continued fraction tends.
For this reason, we term it the {\it{rank}} of the
continued fraction.

It is easy to derive  general divergence results from this theorem, including
Theorem \ref{Stern-Stolz1}, the classical
Stern-Stolz theorem \cite{LW92}. The proof of Theorem \ref{Stern-Stolz1}
is immediate from Theorem \ref{T1}. Just set $\omega_{1}=1$,
$\omega_{2} = -1$ (so $m=2$), $q_{n}=0$ and $p_{n} = b_{n}$.
In fact,
Stern-Stolz can be seen as the beginning of an infinite family
of divergence theorems. We first give a generalization of Stern-Stolz,
then give a corollary describing the infinite family. Last, we  list the first
few examples in the infinite family.

To obtain the generalization, take $q_n=a_n$ instead of $q_n=0$.

\begin{cor}\label{SSgen}
Let the sequences $\{a_n\}$ and $\{b_{n}\}$ satisfy $a_n\ne -1$ for $n\ge 1$,  $\sum |a_{n}| <\infty$ and
$\sum |b_{n}| <\infty$.
Then
\[
b_{0}+K_{n=1}^{\infty}\frac{1+a_n}{b_{n}}
\]
diverges. In fact, for $p=0,1$,
\begin{align*}
&\lim_{n \to \infty}P_{2n+p}=A_{p} \not = \infty,& &\lim_{n \to \infty}Q_{2n+p}=B_{p} \not = \infty,&
\end{align*}
and
\[
A_{1}B_{0}-A_{0}B_{1} = \prod_{n=1}^{\infty}(1+a_n).
\]
\end{cor}
\begin{proof}
This follows immediately from Theorem \ref{T1}, upon setting $\omega_{1}=1$,
$\omega_{2} = -1$ (so $m=2$), $q_{n}=a_n$ and $p_{n} = b_{n}$.
\end{proof}

We have not been able to find Corollary \ref{SSgen} in the literature.

The natural infinite family of Stern-Stolz type theorems is described by the following corollary.

\begin{cor}\label{cor1}
Let the sequences $\{a_n\}$ and $\{b_{n}\}$ satisfy
$a_{n} \not = 1$ for $n\ge 1$, $\sum |a_{n}| <\infty$ and $\sum |b_{n}| <\infty$. Let
$m \geq 3$ and let
$\omega_{1}$ be a primitive $m$-th root of unity. Then
\[
b_{0}+K_{n=1}^{\infty}\frac{-1+ a_{n}}{\omega_{1}+ \omega_{1}^{-1}+b_{n}}
\]
does not converge, but the numerator and denominator convergents in each of the
$m$ arithmetic progressions modulo $m$ do converge.
If $m$ is even, then for $1 \leq p \leq m/2$,
\begin{align*}
&\lim_{n \to \infty}P_{mn+p}=-\lim_{n \to \infty}P_{mn+p+m/2}=A_{p} \not = \infty,& \\
&\lim_{n \to \infty}Q_{mn+p}=-\lim_{n \to \infty}Q_{mn+p+m/2}=B_{p} \not = \infty.&
\end{align*}
If $m$ is odd, then
the continued fraction has rank $m$. If $m$ is even, then
the continued fraction has rank $m/2$. Further, for $2 \leq p \leq m'$, where
$m'=m$ if $m$ is odd and $m/2$ if $m$ is even,
\[
A_{p}B_{p-1}-A_{p-1}B_{p} = -\prod_{n=1}^{\infty}(1-a_n).
\]
\end{cor}
\begin{proof}
In Theorem \ref{T1}, let $\omega_{2} = 1/\omega_{1}$.
\end{proof}
Some explicit examples are given below.

\begin{example}\label{SSseq}
Let the sequences $\{a_n\}$ and $\{b_{n}\}$ satisfy
$a_{n} \not = 1$ for $n\ge 1$, $\sum |a_{n}| <\infty$ and $\sum |b_{n}| <\infty$. Then
each of the following continued fractions diverges:

(i) The following continued fraction has rank three:
 \begin{equation}\label{c3}
b_{0}+K_{n=1}^{\infty}\frac{-1+ a_{n}}{1+b_{n}}.
\end{equation}
 In fact, for $p=1,2,3$,
\begin{align*}
&\lim_{n \to \infty}P_{6n+p}=-\lim_{n \to \infty}P_{6n+p+3}=A_{p} \not = \infty,& \\
&\lim_{n \to \infty}Q_{6n+p}=-\lim_{n \to \infty}Q_{6n+p+3}=B_{p} \not = \infty.&
\end{align*}
(ii) The following continued fraction has rank four:
 \begin{equation}\label{c4}
b_{0}+K_{n=1}^{\infty}\frac{-1+ a_{n}}{\sqrt{2}+b_{n}}.
\end{equation}
In fact, for $p=1,2,3,4$,
\begin{align*}
&\lim_{n \to \infty}P_{8n+p}=-\lim_{n \to \infty}P_{8n+p+4}=A_{p} \not = \infty,& \\
&\lim_{n \to \infty}Q_{8n+p}=-\lim_{n \to \infty}Q_{8n+p+4}=B_{p} \not = \infty.&
\end{align*}
(iii) The following continued fraction has rank five:
\begin{equation}\label{c5}
b_{0}+K_{n=1}^{\infty}\frac{-1+ a_{n}}{(1-\sqrt{5})/2+b_{n}}.
\end{equation}
 In fact, for $p=1,2,3,4,5$,
\begin{align*}
&\lim_{n \to \infty}P_{5n+p}=A_{p} \not = \infty,& &\lim_{n \to \infty}Q_{5n+p}=B_{p} \not = \infty.&
\end{align*}
(iv) The following continued fraction has rank six:
 \begin{equation}\label{c6}
b_{0}+K_{n=1}^{\infty}\frac{-1+ a_{n}}{\sqrt{3}+b_{n}}.
\end{equation}
In fact, for $p=1,2,3,4,5,6$,
\begin{align*}
&\lim_{n \to \infty}P_{12n+p}=-\lim_{n \to \infty}P_{12n+p+6}=A_{p} \not = \infty,& \\
&\lim_{n \to \infty}Q_{12n+p}=-\lim_{n \to \infty}Q_{12n+p+6}=B_{p} \not = \infty.&
\end{align*}
In each case we have, for $p$ in the appropriate range, that
\[
A_{p}B_{p-1}-A_{p-1}B_{p} = -\prod_{n=1}^{\infty}(1-a_n).
\]
\end{example}
\begin{proof}
In Corollary \ref{cor1}, set

(i) $\omega_{1}=\exp ( 2 \pi i/6)$;

(ii) $\omega_{1} = \exp (2 \pi i/8)$;

(iii) $\omega_{1}=\exp ( 2 \pi i/5)$;
(iv)  $\omega_{1}=\exp ( 2 \pi i/12)$.
\end{proof}

The cases $\omega_{1}=\exp ( 2 \pi i/m)$, $m=3,4,10$ give continued fractions
that are the same as those above after an equivalence
transformation and renormalization of the sequences $\{a_n\}$ and$\{b_n\}$.
Note that the continued fractions (\ref{c4}) and (\ref{c6}) are, after an equivalence transformation
and renormalizing the sequences $\{a_n\}$ and $\{b_n\}$, of the forms
\begin{equation}\label{four}
b_{0}+K_{n=1}^{\infty}\frac{-2+ a_{n}}{{2}+b_{n}},
\end{equation}
and
\begin{equation}\label{six}
b_{0}+K_{n=1}^{\infty}\frac{-3+ a_{n}}{{3}+b_{n}},
\end{equation}
respectively.  Because of the equivalence transformations employed, the convergents do not
tend to limits in (\ref{four}) or (\ref{six}). Also, it should be mentioned that Theorem 3.3 of \cite{ABSYZ02} is essentially the
special case $a_n=0$ of part (i) of our example. Nevertheless (\ref{four})
and (\ref{six}) have ranks $4$ and $6$ respectively.

Corollary \ref{c1} now makes it trivial to construct $q$-continued fractions with arbitrarily many limits.

\begin{example}
Let $f(x)$, $g(x) \in \mathbb{Z}[q][x]$ be
polynomials with  zero constant term.
Let $\omega_{1}$, $\omega_{2}$ be distinct roots of unity
and suppose  $m$ is the least positive integer such that $\omega_{1}^m=\omega_{2}^{m}=1$ .
Define
{\allowdisplaybreaks
\begin{equation*}
G(q):=
\frac{-\omega_{1}\omega_{2}+g(q)}{\omega_{1}+\omega_{2}+f(q)}
\+
\frac{-\omega_{1}\omega_{2}+g(q^2)}{\omega_{1}+\omega_{2}+f(q^2)}
\+
\frac{-\omega_{1}\omega_{2}+g(q^3)}{\omega_{1}+\omega_{2}+f(q^3)}
\+
\cds.
\end{equation*}
}
Let $|q|<1$.
If $g(q^n) \not = \omega_{1}\omega_{2}$ for any $n \geq 1$,
then $G(q)$ does not converge.
However,  the sequences of approximants of
$G(q)$ in each of the $m$ arithmetic progressions modulo $m$ converge
to values in $\hat{\mathbb{C}}$. The continued fraction has rank $m/\gcd(b-a,m)$,
where $a$ and $b$ are as defined in Theorem \ref{T1}.
\end{example}

From this example we can conclude that (\ref{RR}) and (\ref{R3}) are far from unique
examples and many other $q$-continued fractions with multiple
limits can be immediately written down. Thus, to Ramanujanize a bit, one can
immediately see that the continued fractions
\begin{equation}
\K{n\ge1}{\infty}\frac{-1/2}{1+q^n}\qquad\text{ and }\qquad \K{n\ge1}{\infty}\frac{-1/2+q^n}{1+q^n}
\end{equation}
both have rank four, while the continued fractions
\begin{equation}
\K{n\ge1}{\infty}\frac{-1/3}{1+q^n}\qquad\text{ and }\qquad \K{n\ge1}{\infty}\frac{-1/3+q^n}{1+q^n}
\end{equation}
both have rank six.

\subsection{Application: Generalization of a Continued Fraction of Ramanujan}

In \cite{BMcL06} we gave a non-trivial example of the preceding
theory, the inspiration for which is a beautiful result of
Ramanujan.

\begin{thm}\label{tmain}
Let $|q|<1$, $|\alpha|=|\beta|=1$, $\alpha \not = \beta$, and the
order of $\lambda:=\alpha/\beta$ in $\mathbb{T}$ be $m$.
For  $x$, $y \not =0$ and fixed $|q|<1$, define
\[
P(x,y)=\sum_{n=0}^{\infty}
\frac{x^{n}q^{n(n+1)/2}}
{(q)_{n}(y\,q)_{n}}.
\]
Then {\allowdisplaybreaks
\begin{multline}\label{ram3gen}
l.s.\left(\frac{ -\alpha \beta }{\alpha +\beta +q} \- \frac{\alpha
\beta}{  \alpha +\beta +q^{2}} \- \frac{\alpha \beta}{  \alpha
+\beta  +q^{3}}
\cds\right)\\
\qquad=-\frac{\beta
P(q\alpha^{-1},\beta\alpha^{-1})\mathbb{T}_m-\alpha
P(q\beta^{-1},\alpha\beta^{-1})} {
P(\alpha^{-1},\beta\alpha^{-1})\mathbb{T}_m-
P(\beta^{-1},\alpha\beta^{-1})}.
\end{multline}
} Moreover,
 {\allowdisplaybreaks\begin{multline}\label{ram3gen2}
\frac{ -\alpha \beta }{\alpha +\beta +q} \- \frac{\alpha \beta}{
\alpha +\beta +q^{2}} \- \frac{\alpha \beta}{  \alpha +\beta
+q^{3}}\cds \- \frac{\alpha \beta}{  \alpha
+\beta  +q^{n}}\\
\sim
-\frac{\beta
P(q\alpha^{-1},\beta\alpha^{-1})\lambda^{n+1}-\alpha
P(q\beta^{-1},\alpha\beta^{-1})} {
P(\alpha^{-1},\beta\alpha^{-1})\lambda^{n+1}-
P(\beta^{-1},\alpha\beta^{-1})}.
\end{multline}
}
\end{thm}

In \cite{BMcL06} we also used the Bauer-Muir Transform to produce
some convergent continued fractions. One such example is the
following.
\begin{corollary}\label{rBM}
Let $|q|<1$ and let $\alpha$ and $\beta$ be distinct points on the
unit circle such that $\alpha/\beta$ is not a root of unity.
 Then
\begin{equation}\label{BMRcf3}
-\beta +\frac{\beta q}{\alpha+q} \+
 K_{n=2}^{\infty}\frac{- \alpha \beta   q}
   { q^n+\alpha +\beta q }
   =-\beta \frac{
   \displaystyle{
   \sum_{n=0}^{\infty}\frac{\alpha^{-n}q^{n(n+3)/2}}
   {(q;q)_{n}(\beta q/\alpha ;q)_{n}} }}
   {\displaystyle{
   \sum_{n=0}^{\infty}\frac{\alpha^{-n}q^{n(n+1)/2}}
   {(q;q)_{n}(\beta q/\alpha ;q)_{n}} }}.
\end{equation}
\end{corollary}

\section{Poincar{\'e} type recurrences}\label{B}
Let the sequence $\{x_{n}\}_{n \geq 0}$ have the initial values
$x_{0}$, $\dots$, $x_{p-1}$ and be subsequently defined by
\begin{equation}\label{pereq}
x_{n+p}=\sum_{r=0}^{p-1}a_{n,r}x_{n+r},
\end{equation}
for $n \geq 0$. Suppose also that there are numbers $a_{0},\dots, a_{p-1}$ such that
\begin{align}
\label{asumineq}
&\lim_{n \to \infty}a_{n,r}=a_{r},& &0 \leq r \leq p-1.&
\end{align}

A recurrence of the form (\ref{pereq}) satisfying the condition (\ref{asumineq}) is
called a Poincar{\'e}-type recurrence, (\ref{asumineq}) being known as the
Poincar{\'e} condition. Such recurrences were initially studied by Poincar{\'e}
who proved that if the roots of the characteristic equation
\begin{equation}\label{chareq}
t^{p}-a_{p-1}t^{p-1}-a_{p-2}t^{p-2}- \dots -a_{0}=0
\end{equation}
have distinct norms, then the ratios of consecutive terms in the
recurrence (for any set of initial conditions) tend to one of the
roots.  See \cite{Poi}.  Because the roots are also the
eigenvalues of the associated companion matrix, they  are also
referred to as the eigenvalues of (\ref{pereq}).  This result was
improved by O. Perron, who obtained a number of theorems about the
limiting asymptotics of such recurrence sequences. Perron
\cite{Perron1} made a significant advance in 1921 when he proved
the following theorem which for the first time treated cases of
eigenvalues which repeat or are of equal norm.

\begin{thm}\label{tp1}
Let the sequence $\{x_{n}\}_{n \geq 0}$ be defined by initial values
$x_{0}$, $\dots$, $x_{p-1}$ and by (\ref{pereq})
for $n \geq 0$. Suppose also that there are numbers $a_{0},\dots, a_{p-1}$ satisfying (\ref{asumineq}).
Let $q_{1},\,q_{2}, \dots q_{\sigma}$ be the distinct moduli of the roots of the characteristic
equation (\ref{chareq})
and let $l_{\lambda}$ be the number of roots whose modulus is $q_{\lambda}$, multiple roots counted
according to multiplicity, so that
\[
l_{1}+l_{2}+\dots l_{\sigma}=p.
\]
Then, provided $a_{n,0}$ be different from zero for $n \geq 0$, the difference equation
\eqref{pereq} has a fundamental system of solutions,
which fall into $\sigma$ classes, such that, for the solutions of the $\lambda$-th class
and their linear combinations,
\[
\limsup_{n \to \infty}  \sqrt[n]{|x_{n}|}=q_{\lambda}.
\]
The number of solutions of the $\lambda$-th class is $l_{\lambda}$.
\end{thm}

Thus when all of the characteristic roots have norm $1$, this theorem
gives that
\[
\limsup_{n  \to \infty}  \sqrt[n]{|x_{n}|}= 1.
\]

Another related paper is \cite{KL} where the authors study products
of matrices and give a sufficient condition for their boundedness. This
is then used to study ``equimodular'' limit periodic continued fractions, which
are limit periodic continued fractions in which the characteristic roots of the associated
$2 \times 2$ matrices are all equal in modulus. The matrix theorem in \cite{KL} can
also be used to obtain results about the boundedness of recurrence sequences.
We study a more specialized situation
here and obtain far more detailed information as a consequence.

Our focus is on the case where the characteristic roots
are distinct numbers on the unit circle. Under a condition stronger
than (\ref{asumineq}) we have a theorem showing that all non-trivial
solutions of such recurrences approach
a limit set in a precisely controlled way. Specifically, our theorem
is:

\begin{thm}\label{t2}
Let the sequence $\{x_{n}\}_{n \geq 0}$ be defined by initial values
$x_{0}$, $\dots$, $x_{p-1}$ and by
\begin{equation}\label{pereqa}
x_{n+p}=\sum_{r=0}^{p-1}a_{n,r}x_{n+r},
\end{equation}
for $n \geq 0$. Suppose also that there are numbers $a_{0},\dots,
a_{p-1}$ such that
\begin{align*}
&\sum_{n=0}^{\infty}|a_{r}-a_{n,r}|<\infty,& &0 \leq r \leq p-1.&
\end{align*}
Suppose further that  the roots of the characteristic equation
\begin{equation}\label{chareqa}
t^{p}-a_{p-1}t^{p-1}-a_{p-2}t^{p-2}- \dots -a_{0}=0
\end{equation}
 are distinct and all on the unit circle, with values, say,
$\alpha_{0}$, $\dots$, $\alpha_{p-1}$. Then there exist complex
numbers $c_0, \dots , c_{p-1}$ such that
\begin{equation}\label{xrecur}
 x_{n} \sim \sum_{i=0}^{p-1}c_{i}\alpha_{i}^{n}.
\end{equation}
\end{thm}

The following corollary, also proved in \cite{BMcL05}, is immediate.

\begin{cor}
Let the sequence $\{x_{n}\}_{n \geq 0}$ be defined by initial values
$x_{0}$, $\dots$, $x_{p-1}$
 and by
\begin{equation*}
x_{n+p}=\sum_{r=0}^{p-1}a_{n,r}x_{n+r},
\end{equation*}
for $n \geq 0$.  Suppose also that there are numbers $a_{0},\dots,
a_{p-1}$ such that
\begin{align*}
&\sum_{n=0}^{\infty}|a_{r}-a_{n,r}|<\infty,& &0 \leq r \leq p-1.&
\end{align*}
Suppose further that  the roots of the characteristic equation
\begin{equation*}
t^{p}-a_{p-1}t^{p-1}-a_{p-2}t^{p-2}- \dots -a_{0}=0
\end{equation*}
are distinct roots of unity, say $\alpha_{0}$, $\dots$,
$\alpha_{p-1}$. Let $m$ be the least positive integer such that, for
all $j \in \{0,1,\dots,p-1\}$, $\alpha_{j}^{m} = 1$.
 Then, for $0 \leq j \leq m-1$, the subsequence
$\{x_{mn+j}\}_{n=0}^{\infty}$ converges. Set $l_j=\lim_{n\to\infty}
x_{nm+j}$, for integers $j\ge 0$. Then the (periodic) sequence
$\{l_j\}$ satisfies the recurrence relation
\begin{equation*}
l_{n+p}=\sum_{r=0}^{p-1}a_{r}l_{n+r},
\end{equation*}
and thus there exist constants $c_0,\cdots,c_{p-1}$ such that
\begin{equation*}
l_n=\sum_{i=0}^{p-1} c_i\alpha_i^n.
\end{equation*}
\end{cor}

\section{Applications to $(r,s)$-matrix continued fractions}

In \cite{LB}, the authors define a generalization of continued
fractions called $(r,s)$-matrix continued fractions. This
generalization unifies a number of generalizations of continued
fractions including ``generalized (vector valued) continued
fractions" and ``G-continued fractions", see \cite{LW92} for
terminology.

Here we show that our results apply to limit periodic
$(r,s)$-matrix continued fractions with eigenvalues of equal
magnitude, giving estimates for the asymptotics of their
approximants so that their limit sets can be determined.

For consistency we closely follow the notation used in \cite{LB} to define $(r,s)$-matrix
continued fractions. Let $M_{s,r}(\mathbb{C})$ denote the set of $s\times r$ matrices
over the complex numbers.  Let $\theta_k$ be a sequence of $n\times n$ matrices over
$\mathbb{C}$. Assume that $r+s=n$.  A $(r,s)$-matrix continued fraction is associated
with a recurrence system of the form $Y_k=Y_{k-1}\theta_k$. The continued
fraction is defined by its sequence of approximants. These are sequences of $s\times r$
matrices defined in the following manner.

Define the function $f: D\in M_n(\mathbb{C})\to
M_{s,r}(\mathbb{C})$ by
\begin{equation}\label{fmateq} f(D)=
B^{-1}A,
\end{equation}
where $B$ is the $s\times s$ submatrix
consisting of the last $s$ elements from both the rows and columns
of $D$, and $A$ is the $s\times r$ submatrix consisting of the
first $r$ elements from the last $s$ rows of $D$.

Then the $k$-th approximant of the $(r,s)$-matrix continued
fraction associated with the sequence $\theta_k$ is defined to be
\begin{equation}\label{skeq}
s_k:=f(\theta_k\theta_{k-1}\cdots\theta_2\theta_1).
\end{equation}

To apply Theorem \ref{DM} to this situation, we endow
$M_{s\times r}(\mathbb{C})$ with a metric by letting the distance
function for two such matrices be the maximum absolute value of
the respective differences of corresponding pairs of elements.
Then, providing that the  $f$ is continuous, (a suitable specialization of)
our theorem can be
applied. (Note that $f$ will be continuous providing that it
exists, since the inverse function of a matrix is continuous when
it exists.)

Let $\lim_{k\to\infty}\theta_k=\theta$, for some $\theta\in M_n(\mathbb{C})$. Then the
recurrence system is said to be of Poincar{\'e} type and the
$(r,s)$-matrix continued fraction is said to be limit periodic. Under our usual condition,
Theorem \ref{DM} can be applied and the following theorem results.

\begin{thm}
Suppose that the condition $\sum_{k\ge
1}||\theta_k-\theta||<\infty$ holds, that the matrix $\theta$ is
diagonalizable, and that the eigenvalues of $\theta$ are all of
magnitude $1$. Then the $k$th approximant $s_k$ has the asymptotic
formula
\begin{equation}\label{sksim}
s_k\sim f(\theta^k F),
\end{equation}
where $F$ is the matrix defined by the convergent product
\[
F:=\lim_{k\to\infty}\theta^{-k}\theta_k\theta_{k-1}\cdots\theta_2\theta_1.
\]
\end{thm}
Note that because of the way that $(r,s)$-matrix continued fractions are defined,
we have taken products in the reverse order than the rest of the paper.

As a consequence of this asymptotic, the limit set can be determined from
\[
l.s.(s_k)=l.s.(f(\theta^k F)).
\]

\section{Conclusion}\label{C}

Because of length restrictions, we have omitted several
corollaries as well as most proofs. Interested readers should consult the author's papers
\cite{BMcL05} and \cite{BMcL06}.

\end{document}